\documentclass[11pt, leqno]{amsart}
\usepackage{amsthm,amsfonts,amsmath, amscd, amssymb, hyperref,amssymb}
\usepackage{fullpage}
\usepackage[pdftex]{graphicx}
\usepackage{enumerate}
\usepackage{cite}
\usepackage[usenames,dvipsnames]{color}
\usepackage{subcaption}
\usepackage{graphicx, upgreek}
\usepackage[margin=1in]{geometry}

\usepackage{tikz-cd}
\usepackage{tikz}
\usetikzlibrary{graphs,decorations.pathmorphing,decorations.markings} 
\usetikzlibrary{arrows}

\usepackage{datetime,mathtools}
\usepackage{xcolor}
\usepackage{hyperref}
\definecolor{darkgreen}{rgb}{0,0.4,0}
\definecolor{BrickRed}{rgb}{0.65,0.08,0}
\hypersetup{colorlinks=true,linkcolor=blue,citecolor=red,filecolor=BrickRed,urlcolor=darkgreen}

\mathtoolsset{showonlyrefs} 

\newcount\m \newcount\n
\def\hours{\n=\time \divide\n 60
	\m=-\n \multiply\m 60 \advance\m \time
	\twodigits\n:\twodigits\m}
\def\twodigits#1{\ifnum #1<10 0\fi \number#1}
\date{\today
}

\numberwithin{equation}{section}

\title{Generalized divisor functions in arithmetic progressions: II}

\author{D. T. Nguyen}
\address{Department of Mathematics \& Statistics, Queen's University, 48 University Ave (Jeffery Hall), Kingston, ON K7L 3N6, Canada.}
\email{d.nguyen@queensu.ca}

\newtheorem{thm}{THEOREM}[section]
\newtheorem{lem}[thm]{LEMMA}

\newtheorem{conjecture}[thm]{Conjecture}

\theoremstyle{definition}

\newcommand{\be}{\begin{equation}}
\newcommand{\ee}{\end{equation}}
\newcommand{\bes}{\begin{equation*}}
\newcommand{\ees}{\end{equation*}}

\def\epsilon{\varepsilon}

\tikzset{degil/.style={
		decoration={markings,
			mark= at position 0.5 with {
				\node[transform shape] (tempnode) {$\backslash$};
			}
		},
		postaction={decorate}
	}
}

\begin{document}

\maketitle

\begin{abstract}
	We obtain a new bound on the second moment of modified shifted convolutions of the generalized 3-fold divisor function, and show that, for applications, the modified version is sufficient.
\end{abstract}

\tableofcontents

\section{Introduction}

Let $\zeta^3(s) = \sum_n \tau_3(n) n^{-s}, \Re(s) > 1$. Determining full asymptotic for the shifted convolutions
	\begin{equation} \label{eq:OG}
		\sum_{1\le n\le N} \tau_3(n) \tau_3(n+h)
	\end{equation}
	for various ranges of $h$ is an important problem in number theory in the past hundred years and is still wide open. We believe, however, that \eqref{eq:OG} is too strong for applications. 
	What we mean by that is the following. Suppose $\mathbf{(\mathrm{\textbf{A}})}$ and $\mathbf{(\mathrm{\textbf{C}})}$ are two statements, possibly conjectures, with $\mathbf{(\mathrm{\textbf{A}})} \implies \mathbf{(\mathrm{\textbf{C}})}$. We say that $\mathbf{(\mathrm{\textbf{A}})}$ is ``too strong" for $\mathbf{(\mathrm{\textbf{C}})}$ if there exists a statement $\mathbf{(\mathrm{\textbf{B}})}$ such that, (i), $\mathbf{(\mathrm{\textbf{B}})}$ is easier to prove than $\mathbf{(\mathrm{\textbf{A}})}$, and, (ii), the following diagram of implications
	
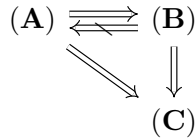
\begin{equation}
		\begin{tikzcd}
			\mathbf{(\mathrm{\textbf{A}})}
			\arrow[Rightarrow]{dr}
			\arrow[Rightarrow, shift left]{r}
			\arrow[Leftarrow, shift right, "\smallsetminus" marking]{r}
			& \mathbf{(\mathrm{\textbf{B}})} \arrow[Rightarrow]{d}
			\\
			& \mathbf{(\mathrm{\textbf{C}})}
		\end{tikzcd}
	\end{equation}
	holds. Thus, we propose the following modified weaker correlation sum
	\begin{equation} \label{eq:new}
			\sum_{n\le X-h} \tau_3(n) \tau_3(n+h),
		\end{equation}
	where the length of the sum depends on the shift $h$. We show that this sum \eqref{eq:new} is close to its expected value in an $L^2$ sense, and that this is enough for certain problems. 
	
	More precisely, we prove, with a power-saving error term, that the second moment of \eqref{eq:new}, namely
	\begin{equation}
		\label{eq:B}
		\tag{$\mathbf{\mathrm{\textbf{B}}}$}
			\sum_{h < X}
			\left(
			\sum_{n\le X-h} \tau_3(n) \tau_3(n+h)
			-
			\mathrm{MT}(X, h)
			\right)^2
			\ll
			X^{3 - 1/100},
		\end{equation}
	is small, for a certain explicit main term $\mathrm{MT}(X, h)$. The main tool used in the proof is a trigonometric method of I. M. Vinogradov. 
	
	Secondly, as an application of the above bound \eqref{eq:B}, we obtain the full asymptotic for the variance of the ternary 3-fold divisor function in arithmetic progressions, averaged over all residue classes (not necessarily coprime) and moduli: There exist computable numerical constants $c_0, \dots, c_8$ such that
	\begin{equation} \label{eq:840}
		\tag{$\mathbf{\mathrm{\textbf{C}}}$}
			\sum_{q \le X}
			\sum_{1 \le a \le q}
			\left(
			\sum_{\substack{n\le X\\ n \equiv a (q)}}
				\tau_3(n)
			-
			\mathrm{MT}(X; q, a)
			\right)^2
			=
			X^2 (c_8 \log^8 X + \cdots + c_0)
			+ 
			O\left( X^{2 - 1/300} \right),
		\end{equation}
	for some explicit main term $\mathrm{MT}(X; q, a)$. This result refines a related conjecture (see conjecture \ref{conj:1} below) about the leading order asymptotic of a similar variance and improves a previously known upper bound by the author.
	
Quantities of the form \eqref{eq:840} have their roots in the celebrated Bombieri-Vinogradov Theorem \cite{Bombieri1965} \cite{Vinogradov1965} (1965), which, in one form, asserts that
\begin{equation} \label{eq:1.1.1}
	\sum_{1\le q\le N^{1/2} (\log N)^{-B}}
	\max_{y\le N} \max_{(a,q)=1}
	\left|\sum_{\substack{1\le n\le y\\ n\equiv a( q)}} \Lambda(n) - \frac{y}{\varphi(y)} \right|
	\ll N(\log N)^{-A}
\end{equation}
where $\Lambda(n)$ is the von Mangoldt function and $B=4A+40$ with $A>0$ arbitrary. Analogues of \eqref{eq:1.1.1} have been found for all $\tau_k(n)$ \cite{NguyenDivisorFunction} and $\tau_2(n)^2$ \cite[Lemma 8]{Motohashi1970}, where $\tau_k(n)$ is the $k$-fold divisor function: $\sum_{n=1}^\infty \tau_k(n) n^{-s} = \zeta^k(s), \Re(s) >1$. Around the same time, Barban \cite{Barban1963} \cite{Barban1964} (1963-1964), Davenport-Halberstam \cite{DavenportHalberstam1966} (1966), and Gallagher \cite{Gallagher1967} (1967) found the following related inequality in which the absolute value is being squared:
\begin{equation} \label{eq:1.2.1}
	\sum_{1\le q\le N(\log N)^{-B}}
	\sum_{\substack{1\le a\le q\\ (a,q)=1}}
	\left| \sum_{\substack{1\le n\le N\\ n\equiv a (q)}} \Lambda(n) - \frac{N}{\varphi(q)} \right|^2
	\ll N^2 (\log N)^{-A},
\end{equation}
giving a much wider range for $q$. In fact, Davenport and Halberstam proved a slightly stronger result than Barban's, while Gallagher gave a simplified elegant proof. For this reason, this type of inequalities are often referred to as Barban-Davenport-Halberstam type inequalities.

Barban-Davenport-Halberstam type inequalities have many applications in number theory. For instance, a version of this inequality (with $\Lambda(n)$ replaced by related convolutions over primes) was skillfully used by Zhang \cite[Lemma 10]{Zhang2014} (2014) in his spectacular work on bounded gaps between primes.

In 1970, Montgomery \cite{Montgomery1970} succeeded in replacing the inequality in \eqref{eq:1.2.1} by an asymptotic equality. Montgomery's method is based on a result of Lavrik, which in turns relied on I. M. Vinogradov's theory of exponential sums over primes. One of Montgomery's results is
\begin{equation} \label{eq:1.2.2}
	\sum_{1\le q\le Q} \sum_{\substack{1\le a\le q\\ (a,q)=1}}
	\left| \sum_{\substack{1\le n\le N\\ n\equiv a (q)}} \Lambda(n) - \frac{N}{\varphi(q)} \right|^2
		= QN \log N + O(QN \log(2N/Q)) + O\left(N^2 (\log N)^{-A} \right)
\end{equation}
for $Q\le N$ and $A>0$ arbitrary. A few years latter, Hooley \cite{HooleyI1975} (1975), by introducing new ideas in treatment of the off-diagonal terms specific to primes, sharpened the right side of \eqref{eq:1.2.2} to
\begin{equation}
	QN \log N + O(QN) + O(N^2 (\log N)^{-A}))
\end{equation}
with $\Lambda(n)$ replaced by the Chebyshev function $\theta(n)$.

Motohashi \cite{Motohashi1973} (1973), by using an approach similar to Montgomery, elaborately established a more precise asymptotic with lower order and power saving error terms for the divisor function $\tau(n)$. Recently, by function field analogues, Rodgers and Soundararajan \cite{RodgersSoundararajan2018} (2018) were led to an analogous conjecture for the leading order asymptotic of the variance of the $k$-fold divisor function $\tau_k$ over the integers. We state here a smoothed version of that conjecture formulated in \cite[Conjecture 1]{NguyenVariance}.

\begin{conjecture} \label{conj:1}
	Let $w(y)$ be a smooth function supported in $[1,2]$ with
	\begin{equation}
		\int w(y)^2 dy = 1,
	\end{equation}
	and
	\begin{equation}
		\mathcal{M}[w](\sigma + it) \ll_\ell \frac{1}{(1+ |t|)^\ell}
	\end{equation}
	uniformly for all $|\sigma| \le A$ for any fixed positive $A>0$, for all positive integers $\ell$, where $\mathcal{M}[w]$ denotes the Mellin transform 
	\begin{equation}
		\mathcal{M}[w](s)
		=
		\int_0^\infty
		w(x) x^{s-1} dx
	\end{equation}
	of $w$.	Then, for $X, d \to \infty$ such that $\frac{\log X}{\log d} \to c \in (0,k)$, we have
	\begin{equation} \label{eq:conj}
		\sum_{\substack{1\le a\le d\\ (a,d)=1}}
		\left|
		\Delta_w(\tau_k;X,d,a)
		\right|^2
		\sim
		a_k(d) \gamma_k(c) X (\log d)^{k^2-1},
	\end{equation}
	where 
	\begin{equation}
		\Delta_w(\tau_k;X,d,a)
		=
		\sum_{n \equiv a \pmod d} \tau_k(n) w\left( \frac{n}{X} \right)
		-
		\frac{1}{\varphi(d)}
		\sum_{(n,d) = 1} \tau_k(n) w\left( \frac{n}{X} \right),
	\end{equation}
		\begin{equation}
				a_k(d) = \lim_{s\to 1^+}
				(s-1)^{k^2}
				\sum_{\substack{n=1\\ (n,d)=1}}^\infty
				\frac{\tau_k(n)^2}{n^s},
			\end{equation}	
		 and $\gamma_k(c)$ is a piecewise polynomial of degree $k^2 - 1$ defined by
		 \begin{equation}
			 	\gamma_k(c) = \frac{1}{k! G(k+1)^2}
			 	\int_{[0,1]^k}
			 	\delta_c (w_1 + \cdots w_k)
			 	\Delta(w)^2 d^kw,
			 \end{equation}
		 with $\delta_c(x) = \delta(x-c)$ a Dirac delta function centered at $c$, $\Delta(w) = \prod_{i<j} (w_i - w_j)$ a Vandermonde determinant, and $G$ the Barnes $G$-function, so that in particular $G(k + 1) =
		 (k - 1)! (k - 2)! \cdots 1!$.
\end{conjecture}

In \cite{RodgersSoundararajan2018}, Rodgers and Soundararajan confirmed an averaged version of this conjecture in a restricted range over smooth cutoffs. Harper and Soundararajan \cite{HarperSound} obtained a lower bound of the right order of magnitude for the average of this variance. In \cite{NguyenVariance}, by using the functional equation for $L(s,\chi)$ and a multiplicative Voronoi summation formula, the author confirmed the asymptotic \eqref{eq:conj} for the restricted dual range $k-1 < c < k$ for all $k$. By the large sieve inequality, the author also obtained in \cite[Theorem 3]{NguyenDivisorFunction} an upper bound of the same order of magnitude for this averaged variance.

The smoothed asymptotic \eqref{eq:conj} as well as the un-smoothed version are closely related to the problem of moments of Dirichlet $L$-functions \cite{ConreyGonnek2002} and correlations of divisor sums \cite{ConreyKeatingI}. This is due to the appearance of the factor $\gamma_k(c)$ in the leading order asymptotic in \eqref{eq:conj}. This piecewise polynomial ``gamma-k-c", as it is commonly refered to, is known to be connected with the geometric constants $g_k$ in the moment conjecture
\begin{equation} \label{eq:moment}
	\int_0^T \left| \zeta\left(\frac{1}{2} + it\right) \right|^{2k} dt
	\sim a_k g_k T\frac{(\log T)^{k^2}}{k^2!}, (T \to \infty),
\end{equation}
where
\begin{equation}
	a_k = 
	\prod_p 
	\left(1-\frac{1}{p}\right)^{(k-1)^2}
	\left(
	1 + \frac{\binom{k-1}{1}^2}{p}
	+ \frac{\binom{k-1}{2}^2}{p^2}
	+ \cdots
	\right)
\end{equation}
by the conjectural relation
\begin{equation} \label{eq:243}
	k^2! \int_0^k \gamma_k(c) dc
	= g_k,\quad
	(k \ge 1).
\end{equation}
We note that the coprimity condition $(a,d)=1$ in \eqref{eq:conj} is essential for this phenomenon.  

In summary, we obtain in this paper a new upper bound for the second moment of the error term of the modified shifted convolution of $\tau_3(n)$ in Theorem \ref{thm:mainresultA} and, as an application, apply this bound to obtain a full asymptotic with a power-saving error term for a variance of $\tau_3(n)$ in arithmetic progressions in Theorem \ref{thm:mainresult}. The novelty of our results is the demonstration that a modified version of the original additive correlation sum is adequate for certain applications.

\subsection{Notations}

$\tau_k(n)$: the number of ways to write a natural number $n$ as an ordered product of $k$ positive integers.

$\tau(n) = \tau_2(n)$: the usual divisor function.

$\varphi(n)$: Euler's function, i.e., the number of reduced residue classes modulo $n$.

$\zeta(s)$: Riemann's zeta function with variable $s=\sigma+i t$.

$\Gamma(s)$: Gamma function.

$\gamma$: Euler's constant $= 0.5722 \dots$.

$\gamma_0(\alpha)$: 0-th generalized Stieltjes constant
\begin{equation}
	\gamma_0(\alpha) = \lim_{m\to \infty}
	\left(\sum_{k=0}^m \frac{1}{k+\alpha} - \log(m+\alpha) \right).
\end{equation}

$e(x)= e^{2\pi i x}.$

$e_q(a) = e^{2\pi i \frac{a}{q}}$.

$c_q(b)$: Ramanujan's sum
\begin{equation}
	c_q(b) = \sum_{\substack{1\le a\le q\\ (a,q)=1}} e_q(ab).
\end{equation}

$(m,n)$: the greatest common divisor of $m$ and $n$.

$[m,n]$: the least common multiple of $m$ and $n$.

$N$: sufficiently large integer.

$\epsilon$: arbitrary small positive constant, not necessarily the same in each occurrence.

$P_r(\log N)$: a polynomial of degree $r$ in $\log N$, not necessarily the same in each occurrence.

Throughout the paper, all constants in $O$-terms or in Vinogradov's notation $\ll$ depends on $\epsilon$ at most.

\section{Statement of results}

Our main results are the following.

\begin{thm} \label{thm:mainresultA}
	We have, for sufficiently large $N$,
	\begin{equation} \label{eq:thm1}
		\sum_{1 \le k <N}
		\left(
		\sum_{ n\le N - k}
		\tau_3(n)
		\tau_3(n + k)
		-
		S_\Delta(k, N)
		\right)^2
		\ll 
		N^{299/100},
	\end{equation}
	where $S_\Delta(k, N)$ is given by \eqref{eq:3.2.7} below with $\Delta = N^{4/19}$.
\end{thm}

As an application, we apply \eqref{eq:thm1} to prove

\begin{thm} \label{thm:mainresult}
	We have the following asymptotic equality, with effectively computable numerical constants $\mathfrak{S}_j,\ (0\le j\le 8)$,
	\begin{align} \label{eq:2.1.1}
		\sum_{1\le \ell \le N}
		\sum_{1\le b\le \ell}
		\left|
			\sum_{\substack{1\le n\le N\\ n\equiv b (\textrm{mod } \ell)}} \tau_3(n)
			- N P_2(\log N) 
		\right|^2
		= N^2 \sum_{j=0}^8 \mathfrak{S}_{8-j} \log^{8-j}N
		+ O\left(N^{2 - \frac{1}{300}
		} \right),
	\end{align}
	where
	\begin{equation} \label{eq:EMT}
			P_2(\log N) 
			= \underset{s=1}{\mathrm{Res}}
			\left\{ 
			\sum_{n\equiv b (\textrm{mod } \ell)} \frac{\tau_3(n)}{n^s} \frac{N^{s-1}}{s}
			\right\}.
		\end{equation}
\end{thm}

\subsection{Remarks} \label{section:remarks}
\begin{enumerate}
	\item Our bound \eqref{eq:thm1} is an improvement of a related result of Baier, Browning, Marasingha, and Zhao \cite{BBMZ}, who proved an analogous estimate to \eqref{eq:thm1} but for correlations of $\tau_3$ with fixed length and shifts up to $N^{1-\epsilon}$. More precisely, they proved in \cite[Theorem 2]{BBMZ}: Assume that $N^{1/3 + \epsilon} \le H \le N^{1-\epsilon}$. Then there exists $\delta >0 $ such that
	\begin{equation}
		\sum_{h\le H}
		\left|
		\sum_{N < n \le 2N}
		\tau_3(n) \tau_3(n+h)
		-
		\mathrm{MT_2}(N, h)
		\right|^2
		\ll
		H N^{2- \delta},
	\end{equation}
	for some main term $\mathrm{MT_2}(N, h)$.
	
	\item The expected main term (EMT) $P_2(\log N)$ in \eqref{eq:2.1.1} is a certain polynomial of degree two in $\log N$ whose coefficients can all be determined explicitly (c.f. Lemma \ref{lemma:EMT}):
	\begin{equation}
		P_2(\log N) 
		= \frac{1}{2} \tilde{A} \log^2 N - (\tilde{A}-\tilde{B})\log N + (\tilde{A}-\tilde{B}+\tilde{C}),
	\end{equation}
	where
	\begin{equation}
		\tilde{A} = \tilde{A}(\ell,b)
		= \ell^{-1} \sum_{q\mid \ell} 
		q^{-3} c_q(b) \sum_{\alpha,\beta,\gamma=1}^q
		e_q(a\alpha \beta\gamma),
	\end{equation}
	\begin{equation}
		\tilde{B} = \tilde{B}(\ell,b)
		= \ell^{-1} \sum_{q\mid \ell} 
		q^{-3} c_q(b) \sum_{\alpha,\beta,\gamma=1}^q
		e_q(a\alpha \beta\gamma)
		(3\gamma_0(\alpha/q) - 3\log q),
	\end{equation}
	\begin{equation}
		\tilde{C} = \tilde{C}(\ell,b)
		= \ell^{-1} \sum_{q\mid \ell} 
		q^{-3} c_q(b) \sum_{\alpha,\beta,\gamma=1}^q
		e_q(a\alpha \beta\gamma)
		(3 \gamma_0(\alpha/q) \gamma_0(\beta/q) - 9\gamma_0(\alpha/q)\log q + \frac{9}{2} \log^2 q),
	\end{equation}
	$\gamma$ is Euler's constant, $\gamma_0(\alpha)$ is the 0-th Stieltjes constant, and $c_q(b)$ is the Ramanujan sum. Different main terms are also considered by other authors. Our choice of EMT \eqref{eq:EMT} here differs from that of \eqref{eq:conj} by an admissible amount which can be shown to be at most $O(X^{2/3+\epsilon})$. This has the harmless effect of changing lower order terms coefficients $\mathfrak{S}_j$ in the asymptotic; see the discussion proceeding Lemma \ref{lemma:Q1} below. Since our average over $b (\textrm{mod } \ell)$ is over all residue classes not necessarily coprime to $\ell$, the expression \eqref{eq:EMT} is the natural EMT to consider, as can readily be seen from its shape. When coprimality condition is imposed on $b (\textrm{mod } \ell)$, the corresponding EMT is the one appearing in \eqref{eq:conj}; this EMT comes from the contribution of the principal character $\chi$ mod $\ell$.
	
	\item The constants $\mathfrak{S}_j, 0\le j\le 8,$ have complicated expressions but can be explicitly determined from our proof. We give here the value of the leading constant $\mathfrak{S}_8$:
	\begin{equation} \label{eq:S8}
		\mathfrak{S}_8 = 
		\frac{1}{8!}
		\prod_p \left(1-9p^{-2} + 16 p^{-3} - 9 p^{-4} + p^{-6} \right)
		\approx 
		1.22326
		\times 10^{-6}.
	\end{equation}
	
	\item T. Parry has recently informed us that he has succeeded in obtaining an asymptotic formula for all $k$ for the quantity
	\begin{equation}
		\sum_{q<Q} \sum_{a=1}^q \left|\sum_{n<x\atop {n=a(q)}} \tau_k(n)- \text{main term} \right|^2,
	\end{equation}
	with power saving error terms.
	
	His result is now available: By using a different method of Goldston and Vaughan  \cite{GoldstonVaughan}, Parry obtained in \cite[Theorem 1]{Parry23} there is a quantity $f_x(q, a)$, such that for fixed $a, q \ge 1$, 
	\begin{equation}
		\sum_{\substack{n \le x\\ n \equiv a \pmod q}}
		\tau_k(n)
		\sim
		\frac{x}{q} f_x(q, a),
		\quad
		(x \to \infty),
	\end{equation}
	and setting
	\begin {eqnarray*}
	E_x(q,a)=\sum _{n\leq x\atop {n\equiv a(q)}}\tau_k(n)-\frac {xf_x(q,a)}{q}\hspace {6mm}\text { and }\hspace {6mm}V(x,Q)&=&\sum _{q\leq Q}\sum _{a=1}^q|E_x(q,a)|^2,
	\end {eqnarray*}
	one has, for some polynomial $P(\cdot ,\cdot )$ of degree $\leq k^2-1$ and $1\leq Q=o(x)$,
	\begin {equation} \label{eq:Parry}
	V(x,Q)=xQP(\log x,\log Q)+\mathcal O_{k,\epsilon }\left (Q^2\left (\frac {x}{Q}\right )^{\mathfrak c}+\underbrace {x^{3/2+\epsilon }}_{k=2}+\underbrace {x^{2-4/(6k-3)+\epsilon }}_{k>2}+Qx^{1-\mathfrak d+\epsilon }\right ),
	\end {equation}
	where
	\begin{equation}
		\mathfrak{c}
		\in
		\begin{cases}
			(1/2 , 1), & \text{ for } k=2,
			\\
			( 1- 1/k(k-2), 1 ), & \text{ for } k>2,
		\end{cases}
	\end{equation}
	and $\mathfrak{d} \in (0,1)$ is any value for which we have
	\begin{equation}
		\sum_{n \le X} \tau_k(n)^2
		=
		X P(\log X) + O_\epsilon\left( X^{1-\mathfrak{d} + \epsilon} \right)
	\end{equation}
	for some polynomial $P$ of degree $k^2-1$. We view our endpoint estimate \eqref{eq:2.1.1} a complement to \eqref{eq:Parry} and vice versa for $k=3$.
\end{enumerate}

\subsection{Outline of the proofs} \label{section:outlineofproof} We follow the approach of Motohashi \cite{Motohashi1973} in his treatment of the divisor function $\tau(n)$, which in turn was based on Montgomery's adaptation \cite{Montgomery1970} of a result of Lavrik \cite{Lavrik1961} on twin primes on average. 

To control the error term, we prove an analog of Lavrik's result for $\tau_3$, using a simpler version of Vinogradov's method of trigonometric sums, as in \cite{Motohashi1973}. The standard convexity bound for $\zeta(s)$ in the critical strip suffices for our purpose. We remark here that our analogue of Lavrik's result can be seen as an average result concerning the mean square error of the following modified additive divisor sum
\begin{equation} \label{eq:modifiedAdditiveDivisorSum}
	\sum_{1\le n\le N-h} \tau_3(n) \tau_3(n+h)
\end{equation}
of length $N-h$ averaged over $h$ up to $h\le N-1$. The advantage of considering \eqref{eq:modifiedAdditiveDivisorSum} is that the length of this sum becomes shorter the larger the shift $h$ is, making contribution from large shifts small, so a power saving is possible when an average over $h$ is taken. This idea might also have applications to the sixth power moment of $\zeta(s)$, which we plan to revisit in the near future. 

To evaluate the main term, we proceed slightly different from Motohashi due to some complications involving an exponential sum in three variables. We show that the resulting sum can be evaluated, on average, thanks to an orthogonality property of the Ramanujan's sums.

\section{Proof of Theorem \ref{thm:mainresultA}}

For $\sigma >1$ and $(a,q)=1$, let
\begin{equation} \label{eq:926}
	E\left(s; \frac{a}{q} \right) = E_3\left(s; \frac{a}{q} \right) = \sum_{n=1}^\infty \tau_3(n) e_q(an) n^{-s}.
\end{equation}
The case for the usual divisor function $\tau(n)$ was considered by Estermann (1930) who obtained analytic continuation and the functional equation for the corresponding generating function. Smith (1982) extended the result to all $\tau_k$. We specialize to a special case his results.
\begin{lem}\cite[Theorem 1, pg. 258]{Smith1982} \label{lemma:1102}
	The function $E_3(s;a/q)$ has a meromorphic continuation to the whole complex plane where it is everywhere holomorphic except for a pole of order 3 at $s=1$. Moreover, $E(s;a/q)$ satisfies the functional equation
	\begin{equation} \label{eq:1038}
		E(s;a/q) = \left(\frac{q}{\pi}\right)^{-\frac{3}{2} (2s-1)}
		\frac{\Gamma^3\left(\frac{1-s}{2}\right)}{\Gamma^3(\frac{s}{2})} E^+(1-s;a/q)
		+ i \left(\frac{q}{\pi}\right)^{\frac{3}{2}(2s-1)}
		\frac{\Gamma^3\left(\frac{2-s}{2}\right)}{\Gamma^3\left(\frac{1+s}{2}\right)} E^-(1-s;a/q),
	\end{equation}
	where
	\begin{equation}
		E^{\pm}(s;a/q) = \sum_{m_1,m_2,m_3\ge 1}
		G^\pm(m_1, m_2, m_3; a/q) (m_1 m_2 m_3)^{-s},\quad (\sigma >1),
	\end{equation}
	\begin{equation}
		G^\pm(m_1,m_2,m_3;a/q) = \frac{1}{2q^{3/2}}
		\{G(m_1,m_2,m_3;a/q) \pm G(m_1,m_2,m_3;-a/q) \},
	\end{equation}
	and
	\begin{equation}
		G(m_1,m_2,m_3; a/q) = \sum_{x_1,x_2,x_3 (q)}
		e_q(am_1m_2m_3 + m_1x_1 + m_2x_2 + m_3x_3).
	\end{equation}
\end{lem}

We rewrite the functional equation \eqref{eq:1038} as follows (c.f. Ivic \cite{Ivic1997}). Let
\begin{align}
	A^{\pm}(n,a/q) = \sum_{n_1n_2n_3=n}
	\sum_{x_1,x_2,x_3=1}^q
	\frac{1}{2} 
	&\big(e_q(a x_1 x_2 x_3 + n_1x_1 + n_2x_2 + n_3x_3) 
	\\&\pm e_q(-a x_1 x_2 x_3 + n_1x_1 + n_2x_2 + n_3x_3) \big).
\end{align}
We have that
\begin{equation}
	|A^{\pm}(n,a/q)| \le q^3 \tau_3(n).
\end{equation}
Then from Lemma \ref{lemma:1102} we obtain the following form of the functional equation.
\begin{lem}\cite[Lemma 2, pg. 1007]{Ivic1997} \label{lemma:1103}
	For $\sigma <0$ and $(a,q)=1$, we have
	\begin{equation} \label{eq:326}
		E(s; a/q) = \left(\frac{q}{\pi}\right)^{-\frac{3}{2} (2s-1)}
		\left\{\frac{\Gamma^3\left(\frac{1-s}{2}\right)}{\Gamma^3\left(\frac{s}{2}\right)}
		\sum_{n=1}^\infty A^+(n,a/q) n^{s-1}
		+ i \frac{\Gamma^3\left(\frac{2-s}{2}\right)}{\Gamma^3\left(\frac{1+s}{2} \right)}
		\sum_{n=1}^\infty A^-(n,a/q) n^{s-1}
		\right\},
	\end{equation}
	where the two series on the right-side are absolutely convergent.
\end{lem}

We also need the Laurent expansion of $E(s; a/q)$ at $s=1$ for residue calculations. We first recall a lemma from Motohashi \cite{Motohashi1973}.

\begin{lem} \label{lemma:taudm}
	We have, uniformly for any integer $d$,
	\begin{equation} \label{eq:taudm}
		\sum_{1\le m\le y} \tau(dm)
		= y \sum_{q \mid d} \frac{\varphi(q)}{q}
		(\log dy + 2\gamma -1 - 2\log q)
		+ O(\tau^2(d) y^{1/2} \log^2 y).
	\end{equation}
\end{lem}
\begin{proof}
	See \cite[Lemma 4.6.1, p. 193]{Motohashi1973}.
\end{proof}

\begin{lem} \label{lemma:309}
	For $(a,q) =1$, we have
	\begin{equation} \label{eq:327}
		E(s;a/q) = \frac{1}{q} \left(\frac{A}{(s-1)^3} + \frac{B}{(s-1)^2} + \frac{C}{s-1}\right)
		+ \sum_{n=0}^\infty c_n(a,q) (s-1)^n,
	\end{equation}
	where
	\begin{align}
		A &= A(q) = q^{-2}\sum_{\alpha,\beta,\gamma=1}^q
		e_q(a\alpha \beta\gamma),\\
		B &= B(q) = q^{-2}\sum_{\alpha,\beta,\gamma=1}^q
				e_q(a\alpha \beta\gamma)
				(3\gamma_0(\alpha/q) - 3\log q),\\
		C &= C(q) = q^{-2}\sum_{\alpha,\beta,\gamma=1}^q
				e_q(a\alpha \beta\gamma)
				(3 \gamma_0(\alpha/q) \gamma_0(\beta/q) - 9\gamma_0(\alpha/q)\log q + \frac{9}{2} \log^2 q),
	\end{align}
	with
	\begin{equation} \label{eq:GaussDigamma}
		\gamma_0(\alpha) = \lim_{m\to \infty}
		\left(\sum_{k=0}^m \frac{1}{k+\alpha} - \log(m+\alpha) \right).
	\end{equation}
	The coefficients $A,B,C$ are independent of $a$ and satisfy
	\begin{align} \label{eq:103}
		A(q) \ll \log^2q,\quad
		B(q) \ll \log^3 q,\quad
		C(q) \ll \log^4 q,
	\end{align}
	uniformly in $a$.
\end{lem}
\begin{proof}
	See Ivi\'c \cite[pp. 1007-1008]{Estermann1930} for the Laurent expansion \eqref{eq:327}. In fact, Ivi\'c also gave upper bound of the form $q^\epsilon$ for $A,B$, and $C$ which came from the bound
	\begin{equation}
		\sum_{\alpha,\beta,\gamma=1}^q
				e_q(a\alpha \beta\gamma)
		\le q \sum_{1\le \alpha\le q} \tau( \alpha q)
		\ll q^{2+\epsilon}.
	\end{equation}
	We can sharpen this upper bound by applying Lemma \ref{lemma:taudm} and bounding $\sum_{q \mid d} \frac{\varphi(q)}{q}$ by $\log(q)$, giving
	\begin{equation}
			\sum_{\alpha,\beta,\gamma=1}^q
					e_q(a\alpha \beta\gamma)
			\le q \sum_{1\le \alpha\le q} \tau( \alpha q)
			\ll q^2 \log^2q.
	\end{equation}
	Thus, noting that $\gamma_0(\alpha) \ll 1$ by \eqref{eq:GaussDigamma}, the bound \eqref{eq:103} follows.
\end{proof}

\begin{lem}
	For $n\ge 1$ and $(a,q)=1$, we have
	\begin{equation} \label{eq:524}
		\underset{s=1}{\mathrm{Res}}\ E\left(s;\frac{a}{q}\right) \frac{n^s}{s}
		= q^{-1} n \left(\frac{A}{2}\log^2n - (A-B)\log n + (A-B+C) \right),
	\end{equation}
	where $A,B,C$ are given in Lemma \ref{lemma:309}.
\end{lem}
\begin{proof}
	We have, by \eqref{eq:327},
	\begin{align}
		\underset{s=1}{\mathrm{Res}}\ E\left(s;\frac{a}{q}\right) \frac{n^s}{s}
		&= \frac{1}{2} \lim_{s\to1}
		\frac{d^2}{ds^2}
		\left((s-1)^3 E\left(s;\frac{a}{q}\right) \frac{n^s}{s} \right)
		\\ &= \frac{1}{2q} \lim_{s\to1}
				\frac{d^2}{ds^2}
				\left((A + B(s-1) + C(s-1)^2 + O((s-1)^3)) \frac{n^s}{s} \right)
		\\ &= q^{-1} n \left(\frac{A}{2}\log^2n - (A-B)\log n + (A-B+C) \right).
	\end{align}
\end{proof}

\begin{lem} \label{lemma:EMT}
	For $\sigma >1$, let
	\begin{equation}
		R(s;\ell,b) = \sum_{n\equiv b (\textrm{mod } \ell)} \tau_3(n) n^{-s}.
	\end{equation}
	We have
	\begin{equation}
		\underset{s=1}{\mathrm{Res}}\ 
		R(s;\ell,b) \frac{N^s}{s}
		= N \left(\frac{\tilde{A}}{2}\log^2 N - (\tilde{A}-\tilde{B})\log N + (\tilde{A}-\tilde{B}+\tilde{C}) \right),
	\end{equation}
	where
	\begin{align}
		\tilde{A} &= \tilde{A}(\ell,b)
		= \ell^{-1} \sum_{q\mid \ell} 
		q^{-1} c_q(b) A(q), \\
		\tilde{B} &= \tilde{B}(\ell,b)
				= \ell^{-1} \sum_{q\mid \ell} 
				q^{-1} c_q(b) B(q),
		\\ \tilde{C} &= \tilde{C}(\ell,b)
				= \ell^{-1} \sum_{q\mid \ell} 
				q^{-1} c_q(b) C(q),
	\end{align}
	with $A(q),B(q),C(q)$ given in Lemma \ref{lemma:309}.
\end{lem}
\begin{proof}
	We can write $R(s;\ell,b)$ as
	\begin{align}
		R(s;\ell,b) 
		&= \frac{1}{\ell} 
		\sum_{q \mid \ell}
		\sum_{\substack{1\le a\le q\\ (a,q)=1}}
		e_q(-ab) E\left(s;\frac{a}{q} \right)
		\\ &= \frac{1}{\ell} 
		\sum_{q \mid \ell} 
		\frac{1}{q} c_q(b)
		\left(\frac{A(q)}{(s-1)^3} + \frac{B(q)}{(s-1)^2} + \frac{C(q)}{s-1}\right)
		+ \sum_{n=0}^\infty \frac{1}{\ell} 
				\sum_{q \mid \ell} 
				\frac{1}{q} c_q(b)
				 c_n(a,q) (s-1)^n
		\\ &= \frac{\tilde{A}(\ell,b)}{(s-1)^3} + \frac{\tilde{B}(\ell,b)}{(s-1)^2} + \frac{\tilde{C}(\ell,b)}{s-1}
		+ \sum_{n=0}^\infty \frac{1}{\ell} 
						\sum_{q \mid \ell} 
						\frac{1}{q} c_q(b)
						 c_n(a,q) (s-1)^n.
	\end{align}
	The lemma follows as in the previous one.
\end{proof}

For $\alpha\in \mathbb{R}$, let
\begin{equation} \label{eq:535}
	D(\alpha,N) = \sum_{1\le n\le N} \tau_3(n) e(\alpha n).
\end{equation}
Using \eqref{eq:926} we first estimate $D(\alpha,N)$ for $\alpha=a/q$ with $(a,q)=1$.
\begin{lem} \label{lemma:314}
	For $(a,q)=1$, we have
	\begin{equation}
		D\left(\frac{a}{q},n \right)
		= \frac{n}{q} \left( \frac{A}{2} \log^2n - (A-B)\log n + (A-B+C) \right)
		+ O\left\{(nq+q^2)^{3/5+\epsilon} \right\},
	\end{equation}
	with $A,B,C$ given in Lemma \ref{lemma:309}.
\end{lem}
\begin{proof}
	We have
	\begin{align} \label{eq:520}
		D\left(\frac{a}{q},n \right)
		&= \underset{s=1}{\mathrm{Res}}\
		E\left(s;\frac{a}{q}\right) \frac{n^s}{s}
		+ \underset{s=0}{\mathrm{Res}}\
				E\left(s;\frac{a}{q}\right) \frac{n^s}{s}
		+ \frac{1}{2\pi i} \int_{-\delta-iT}^{-\delta+iT}
		E\left(s;\frac{a}{q}\right) \frac{n^s}{s} ds
		\\ &+ O\left\{\frac{n^{1+\epsilon}}{T} + n^\epsilon + \frac{1}{T} \int_{-\delta}^{1+\delta} \left|E\left(\sigma+iT;\frac{a}{q}\right)\right| n^\sigma d\sigma \right\},
	\end{align}
	where $\delta= (\log(nq+1))^{-1}$ and $T$ is to be determined latter. By expressing the residue as an integral around the origin,
	\begin{equation} \label{eq:521}
		\left|\underset{s=0}{\mathrm{Res}}\
						E\left(s;\frac{a}{q}\right) \frac{n^s}{s} \right|
		\ll (\log(qn+1))^3.
	\end{equation}
	By the functional equation \eqref{eq:326} and the convexity argument,
	\begin{equation}
		\left|E\left(\sigma+iT; \frac{a}{q}\right) \right|
		\ll (qT)^{\frac{3}{2}(1-\sigma)} (\log qT)^6
	\end{equation}
	uniformly for $-\delta \le \sigma \le 1+\delta$. Hence we get
	\begin{equation} \label{eq:522}
		\left|\frac{1}{2\pi i} \int_{-\delta-iT}^{-\delta+iT}
				E\left(s;\frac{a}{q}\right) \frac{n^s}{s} ds \right|
		\ll (Tq)^\frac{3}{2} (\log qT)^7
	\end{equation}
	and
	\begin{equation} \label{eq:523}
		\frac{1}{T}\int_{-\delta}^{1+\delta} \left|E\left(\sigma+iT;\frac{a}{q}\right)\right| n^\sigma d\sigma
		\ll \frac{n}{T} (\log qT)^6
		\int_{-\delta}^{1+\delta} \left(\frac{Tq}{n^{2/3}}\right)^{\frac{3}{2}(1-\sigma)} d\sigma.
	\end{equation}
	Taking
	\begin{equation}
		T = q^{-1} (nq + q^2)^{2/5}
	\end{equation}
	it follows from \eqref{eq:524}, \eqref{eq:520}, \eqref{eq:521}, \eqref{eq:522} and \eqref{eq:523} that
	\begin{equation}
		D\left(\frac{a}{q},n \right)
				= \frac{n}{q} \left( \frac{A}{2} \log^2n - (A-B)\log n + (A-B+C) \right)
				+ O\left\{(nq+q^2)^{3/5+\epsilon} \right\}.
	\end{equation}
\end{proof}

\begin{lem} \label{lemma:528}
	For $\alpha \in \mathbb{R}$, we have
	\begin{align} \label{eq:537}
		D(\alpha,N) 
		&= \frac{1}{q}
		\sum_{1\le n\le N} 
		\left( \frac{A}{2} \log^2n - (A-B)\log n + (A-B+C) \right)
		e\left(\left(\alpha-\frac{a}{q}\right)n \right)
		\\&+ O\left\{(Nq+q^2)^{3/5+\epsilon} \left(1+\left|\alpha-\frac{a}{q} \right| N \right) \right\},
	\end{align}
	with $A,B,C$ given in Lemma \ref{lemma:309}.
\end{lem}
\begin{proof}
	We have
	\begin{equation}
		D(\alpha,N) = \sum_{1\le n\le N}
			\left\{D(a/q,n) - D(a/q,n-1) \right\}
			e((\alpha-a/q)n).
	\end{equation}
	This, together with Lemma \ref{lemma:314} and partial summation, gives \eqref{eq:537}.
\end{proof}

Let
	\begin{equation} \label{eq:F}
		F\left(\alpha, \frac{a}{q}, N \right)
		= \frac{1}{q}
				\sum_{1\le n\le N} 
				\left( \frac{A}{2} \log^2n - (A-B)\log n + (A-B+C) \right)
				e\left(\left(\alpha-\frac{a}{q}\right)n \right)
	\end{equation}
	and
	\begin{equation} \label{eq:GDelta}
		G_\Delta(\alpha,N)
		= \sum_{1\le q\le \Delta}
		\sum_{\substack{a=1\\(a,q)=1}}^q
		\left|F\left(\alpha,\frac{a}{q}, N\right) \right|^2,
	\end{equation}
	where $\Delta$ satisfies
	\begin{equation} \label{eq:3.2.3}
		4\Delta \le \Omega,
	\end{equation}
	where $\Omega$ is the order of a Farey series (see section \ref{section:Lavrik})
	and $\Delta$ is to be determined more precisely later; see \eqref{eq:OmegaDelta} below. By Lemma \ref{lemma:528} and equation \eqref{eq:F},
	\begin{equation} \label{eq:3.2.1}
		\left|D(\alpha,N) -  F(\alpha,a/q,N) \right|
		\ll (Nq+q^2)^{3/5+\epsilon} \left(1+\left|\alpha-\frac{a}{q} \right| N \right).
	\end{equation}
	Now, by \eqref{eq:F} and \eqref{eq:GDelta},
	\begin{equation} \label{eq:3.2.4}
		G_\Delta(\alpha,N) = \sum_{|k|\le N-1} e(\alpha k)
		\left(\sum_{1\le q\le \Delta} \frac{1}{q^2} W_q(k,N) 
		\sum_{\substack{a=1\\(a,q)=1}}^q e_q(-ak) \right),
	\end{equation}
	where
	\begin{align} \label{eq:W_q}
		W_q(k,N) 
		&= \frac{1}{4} A^2 \sum_{1\le n\le N-|k|} \log^2 n \log^2(n+|k|)
		\\ &-\frac{1}{2} A(A-B) \sum_{1\le n\le N-|k|} \log n \log(n+|k|) \log n(n+|k|)
		\\ &+ (A-B)^2 \sum_{1\le n\le N-|k|} \log n \log(n+|k|)
		\\ &-\frac{1}{2} A(A-B+C) \sum_{1\le n\le N-|k|} (\log^2 n + \log^2(n+|k|))
		\\ &- (A-B)(A-B+C) \sum_{1\le n\le N-|k|} \log n(n+|k|)
		\\ &+ (A-B+C)^2 (N-|k|)
		\\&= w_1(q) T_1(k,N) + \cdots + w_6(q) T_6(k,N),
	\end{align}
	say. For the innermost sum in \eqref{eq:3.2.4} we have
	\begin{equation}
		\sum_{\substack{a=1\\(a,q)=1}}^q e_q(-ak)
		= \mu\left(\frac{q}{(q,|k|)}\right)
		\frac{\varphi(q)}{\varphi\left(\frac{q}{(q,|k|)}\right)} = c_q(|k|).
	\end{equation}
	Thus we write \eqref{eq:3.2.4} as
	\begin{equation} \label{eq:3.2.7}
		G_\Delta(\alpha,N)
		= \sum_{|k|\le N-1} 
		\left(\sum_{1\le q\le \Delta} q^{-2} c_q(|k|) W_q(k,N) \right) e(\alpha k)
		= \sum_{|k|\le N-1} S_\Delta(k,N) e(\alpha k),
	\end{equation}
	say. Now, by \eqref{eq:535}, we have
	\begin{equation} \label{eq:3.1.1}
		|D(\alpha,N)|^2 = \sum_{|k| \le N-1} V(k,N) e(\alpha k),
	\end{equation}
	where
	\begin{equation} \label{eq:V(k,N)}
		V(k,N) = \sum_{1\le n\le N-|k|} \tau_3(n) \tau_3(n+|k|).
	\end{equation}
	Thus,
	\begin{equation}
		|D(\alpha,N)|^2 - G_\Delta(\alpha,N)
		= \sum_{|k| \le N-1}
		(V(k,N) - S_\Delta(k,N))
		e(\alpha k).
	\end{equation}
	and we obtain
\begin{lem} \label{lemma:818}
	\begin{equation} \label{eq:3.2.8}
		\sum_{|k| \le N-1} \left(V(k,N) - S_\Delta(k,N) \right)^2
		= \int_0^1 \left| |D(\alpha,N)|^2 - G_\Delta(\alpha,N) \right|^2 d\alpha,
	\end{equation}
	with $D(\alpha,N)$, $G_\Delta(\alpha,N)$, $V(k,N)$, and $S_\Delta(k,N)$ given by \eqref{eq:535}, \eqref{eq:GDelta}, \eqref{eq:V(k,N)}, and \eqref{eq:3.2.7}, respectively.
\end{lem}
This integral will be estimated in Section \ref{section:Lavrik} below.

\begin{lem} \label{lemma:T_j}
	With
	\begin{align} \label{eq:T_j}
		T_1(k,N) &= \sum_{1\le n\le N-|k|} \log^2 n \log^2(n+|k|),\\
		T_2(k,N) &= \sum_{1\le n\le N-|k|} \log n \log(n+|k|) \log (n(n+|k|)),\\
		T_3(k,N) &= \sum_{1\le n\le N-|k|} \log n \log(n+|k|),\\
						T_4(k,N) &= \sum_{1\le n\le N-|k|} (\log^2 n + \log^2(n+|k|)),\\
						T_5(k,N) &= \sum_{1\le n\le N-|k|} \log (n(n+|k|)).
	\end{align}
	given from \eqref{eq:W_q}, we have
	\begin{align}
		T_1(k,N) &= (N-|k|) \log^2N \log^2(N-|k|) + O(N \log^3 N),\\
		T_2(k,N) &= (N-|k|)(\log^2N \log(N-|k|) + \log N \log^2(N-|k|)) +O(N\log^2 N),\\
		T_3(k,N) &= (N-|k|)\log N \log(N-|k|) + O(N\log N),\\
				T_4(k,N) &= (N-|k|)(\log^2N + \log^2(N-|k|)) +O(N\log N),\\
				T_5(k,N) &= (N-|k|)(\log N + \log(N-|k|)) +O(N).
	\end{align}
\end{lem}
\begin{proof}
	For $k>0$, by partial summation, we have
	\begin{equation}
		T_5(k,N) = (N-k) \log(N-k) + N\log N - k\log k - 2(N-k) +O(\log N).
	\end{equation}
	Similarly, we obtain the other $T_j$'s.
\end{proof}

\begin{lem} \label{lemma:cqdeltam}
	We have, for any $(a,q)=1$ and $\epsilon >0$,
	\begin{equation}
		\sum_{n \le X}
		\tau(n) e_q(a n)
		=
		q^{-1}
		X \left(  \log X -2 \log q + 2 \gamma -1 \right)
		+
		O_{\epsilon} \left( (qX)^{\frac{1}{2} + \epsilon} + q^{1+\epsilon} \right),
	\end{equation}
	where the big-oh is independent of $a$.
\end{lem}
\begin{proof}
	See, e.g., \cite[page 179, line -3]{Motohashi1973}.
\end{proof}

We will apply Perron's formula in the following form.

\begin{lem} \label{lemma:Perron}
	Let $f(s) = \sum_{n=1}^\infty a_n n^{-s}$ be a Dirichlet series which converges absolutely for $\sigma>1$. Suppose $a_n = O(n^\epsilon)$ for any $\epsilon>0$ and $f(s) = \zeta(s)^\ell F(s)$ for some natural number $\ell$ and some Dirichlet series $F(s)$ which converges absolutely in $\Re(s) > 1/2$. Then for $X$ not an integer, we have
	\begin{equation} \label{eq:Perron}
		\sum_{n\le X} a_n
		= \frac{F(1)}{(\ell -1)!} X P_{\ell-1}(\log X) + O_\epsilon \left(X^{1-\frac{1}{\ell +2}}\right),
	\end{equation}
	where $P_{\ell-1}(\log X)$ is the polynomial in $\log X$ of degree $\ell - 1$ with leading coefficient 1 given explicitly by
		\begin{equation} \label{eq:polyP}
			P_{\ell-1}(\log X)
			= (\ell - 1)! Res_{s=1} \zeta(s)^\ell F(s) \frac{X^{s-1}}{s}.
		\end{equation}
\end{lem}
\begin{proof}
	See \cite[Problems 4.4.16, 4.4.17]{MurtyBook2007}.
\end{proof}

\begin{lem} \label{lemma:tau3squared}
	We have
	\begin{equation}
		\sum_{n\le X} \tau_3^2(n)
		= \frac{A_3}{8!} X P_8(\log X) + O\left(X^{10/11}\right),
	\end{equation}
	where $A_3 = 8! \mathfrak{S}_8$ with $\mathfrak{S}_8$ given in \eqref{eq:S8}, and $P_8(\log X)$ is a polynomial of degree 8 in $\log X$ and leading coefficient 1.
\end{lem}
\begin{proof}
	We have
	\begin{equation}
		\sum_{n=1}^\infty \tau_3^2(n) n^{-s}
		= \prod_p \left\{1 + \sum_{\nu=1}^\infty \binom{\nu + 2}{2}^2 p^{-\nu s} \right\},
	\end{equation}
	where both members of this equation are absolutely convergent if $\sigma >1$. Hence, if $\sigma >1$,
	\begin{align}
		\left\{\zeta(s) \right\}^{-9}
		\left\{\sum_{n=1}^\infty \tau_3^2(n) n^{-s} \right\} 
		&= \prod_p \left\{(1-p^{-s})^9
		( 1 + 9 p^{-s} + 36 p^{-2s} + \cdots ) \right\}
		\\&= \prod_p \left\{1 + a_2 p^{-2s} + a_3 p^{-3s} + \cdots \right\}
		= F(s),
	\end{align}
	say, where
	\begin{equation}
		a_\nu
		= \sum_{r=0}^\nu
		(-1)^r \binom{9}{r} \binom{\nu - r + 2}{2}^2.
	\end{equation}
	We adopt the convention for the binomial coefficients that $\binom{n}{m} = 0$ if $m>n$. We have $a_2 = -9, a_3 = 16, a_4 = -9, a_5 = 0, a_6 = 1$ and $a_\nu=0$ for $\nu \ge 7$. The coefficient $a_\nu$ satisfies
	\begin{equation}
		|a_\nu| \le K \nu^2,
	\end{equation}
	where $K$ is independent of $\nu$. Hence
	\begin{equation}
		\sum_{\nu=2}^\infty
		|a_\nu| p^{-\nu s}
		\le K' p^{-2s},
	\end{equation}
	where $K'$ is independent of $p$. Hence, if $\sigma >1/2$, then $\sum_p p^{-2s}$ is absolutely convergent, and thus is also
	\begin{equation}
		F(s) = \prod_p \left\{1 + \sum_{\nu=2}^\infty a_\nu p^{-\nu s} \right\}.
	\end{equation}
	Hence we obtain that
	\begin{equation}
		\sum_{n=1}^\infty \tau_3^2(n) n^{-s}
		= \left\{\zeta(s) \right\}^{9} F(s),
	\end{equation}
	where $F(s)$ is absolutely convergent for $\sigma > 1/2$. It follows at once, by Lemma \ref{lemma:Perron}, that
	\begin{equation}
		\sum_{n\le X} \tau_3^2(n)
				= \frac{A_3}{8!} X P_8(\log X) + O\left(X^{10/11}\right),
	\end{equation}
	where $$A_3 = F(1) = \prod_p \left(1-9p^{-2} + 16 p^{-3} - 9 p^{-4} + p^{-6} \right).$$
\end{proof}

\begin{lem} \label{lem:geometricIntegral}
	We have
	\begin{equation}
		\int_1^{N-1} \frac{t \log t}{N-t} dt
		= N\left(\log^2 N - \log N - \frac{\pi^2}{6} + 1 \right) + O(\log N)
	\end{equation}
	and
	\begin{equation}
		\int_1^{N-1} \frac{t \log^2 t}{N-t} dt
				= N\left(\log^3 N - 2 \log^2N -\left(\frac{\pi^2}{3} - 2\right) \log N + 2\zeta(3) - 2 \right) + O(\log^2 N).
	\end{equation}
\end{lem}
\begin{proof}
	Expanding into a geometric series and integrate by parts, we have
	\begin{align}
			\int_1^{N-1} \frac{t \log t}{N-t} dt
			&= \sum_{m=1}^\infty \frac{1}{N^{m}}
				\int_1^{N-1} t^m \log t dt
			\\&= N \log(N-1) \sum_{m=1}^\infty \frac{1}{m+1} \left(\frac{N-1}{N} \right)^{m+1} 
			- N \sum_{m=1}^\infty \frac{1}{(m+1)^2} \left(\frac{N-1}{N} \right)^{m+1} +O(1)
			\\&= N\left(\log^2 N - \log N - \frac{\pi^2}{6} + 1 \right) + O(\log N).
		\end{align}
	This gives the first integral. The second integral is computed in a similar way.
\end{proof}

\subsection{An analogue to a result of Lavrik} \label{section:Lavrik}

In this section we estimate the integral in \eqref{eq:3.2.8} by the trigonometric method of I.M. Vinogradov along the line of Lavrik, following Motohashi (section 3).

Let $a/q$ be a term of the Farey series of order $\Omega$, which is to be determined later; see \eqref{eq:OmegaDelta} below. Let
\begin{equation}
	\frac{a'}{q'}, \frac{a}{q}, \frac{a''}{q''}
\end{equation}
be consecutive terms of the Farey series and let $C(a/q)$ be the interval $\left[\frac{a'+a}{q'+q}, \frac{a+a''}{q+q''} \right]$. The interval $C(a/q)$ contains the fraction $a/q$ with length bounded by
\begin{equation} \label{eq:3.3.1}
	\left|C\left(\frac{a}{q}\right) \right|
	\le \frac{2}{q\Omega}.
\end{equation}
Let
\begin{equation} \label{eq:825}
	U(N) = \int_0^1 \left| |D(\alpha,N)|^2 - G_\Delta(\alpha,N) \right|^2 d\alpha
\end{equation}
denote the integral in \eqref{eq:3.2.8}. We proceed to estimate $U(N)$. We have
\begin{align} \label{eq:3.3.2}
	U(N) 
	&= \sum_{1\le q\le \Omega} \sum_{\substack{a=1\\(a,q)=1}}^q
	\int_{C(a/q)} \left||D(\alpha,N)|^2 - G_\Delta(\alpha,N) \right|^2 d\alpha
	\\&\le 2 \sum_{1\le q\le \Omega} \sum_{\substack{a=1\\(a,q)=1}}^q
		\int_{C(a/q)}
	\left||D(\alpha,N)|^2 - \left|F\left(\alpha,\frac{a}{q},N \right) \right|^2 \right|^2 d\alpha
	\\&+ 2 \sum_{1\le q\le \Omega} \sum_{\substack{a=1\\(a,q)=1}}^q
		\int_{C(a/q)}
	\left|G_\Delta(\alpha,N) - \left|F\left(\alpha,\frac{a}{q},N\right)  \right|^2 \right|^2 d\alpha
	\\&= U_1(N) + U_2(N),
\end{align}
say. For $U_1(N)$, we have, from \eqref{eq:3.2.1} and the inequality $\left| |A|^2 - |B|^2 \right|^2 \le 4 |A-B|^2 ( |A|^2 + |B|^2 )$, valid for any complex numbers $A$ and $B$, that
\begin{equation}
	\left||D(\alpha,N)|^2 - \left|F\left(\alpha,\frac{a}{q},N \right) \right|^2 \right|^2
	\ll (Nq+q^2)^{\frac{6}{5} + 2\epsilon} \left(1+\left|\alpha-\frac{a}{q} \right|^2 N^2 \right)
	\left(|D(\alpha,N)|^2 + \left|F\left(\alpha,\frac{a}{q},N \right) \right|^2 \right).
\end{equation}
Thus, for $\alpha\in C(a/q)$, we have, by \eqref{eq:3.3.1}, that the above is bounded by
\begin{align}
	\left( (N\Omega)^{\frac{6}{5} + 2\epsilon}
	+ \Omega^{\frac{12}{5} + 4\epsilon}
	+ \frac{N^{\frac{16}{5} + 2\epsilon}}{\Omega^2} 
	+ \frac{N^2}{\Omega^{\frac{8}{5} - 4\epsilon}}
	\right)
	\left(|D(\alpha,N)|^2 + \left|F\left(\alpha,\frac{a}{q},N \right) \right|^2 \right),
\end{align}
and we get
\begin{align} \label{eq:3.3.3}
	U_1(N)
	&\ll \left( (N\Omega)^{\frac{6}{5} + 2\epsilon}
		+ \Omega^{\frac{12}{5} + 4\epsilon}
		+ \frac{N^{\frac{16}{5} + 2\epsilon}}{\Omega^2} 
		+ \frac{N^2}{\Omega^{\frac{8}{5} - 4\epsilon}}
		\right)
	\\ &  \quad \quad \times
	\left\{\int_0^1 |D(\alpha,N)|^2 d\alpha
	+ \sum_{1\le q\le \Omega} \sum_{\substack{a=1\\(a,q)=1}}^q \int_0^1 \left|F\left(\alpha,\frac{a}{q},N \right) \right|^2 d\alpha \right\}
	\\&\ll \left( (N\Omega)^{\frac{6}{5} + 2\epsilon}
			+ \Omega^{\frac{12}{5} + 4\epsilon}
			+ \frac{N^{\frac{16}{5} + 2\epsilon}}{\Omega^2} 
			+ \frac{N^2}{\Omega^{\frac{8}{5} - 4\epsilon}}
			\right)
		N \log^8 N.
\end{align}

For $U_2(N)$, we have, by \eqref{eq:GDelta},
\begin{align} \label{eq:3.3.4}
	U_2(N)
	&\ll \sum_{1\le q\le \Omega}
	\sum_{\substack{a=1\\(a,q)=1}}^q
	\int_{C(a/q)}
	\left|\sum_{1\le q'\le \Delta}
	\sum_{\substack{a'=1\\(a',q')=1\\ a'q\neq aq'}}^{q'}
	\left|F\left(\alpha,\frac{a'}{q'},N \right) \right|^2 \right|^2 d\alpha
	\\&+ \sum_{\Delta < q \le \Omega} 
	\sum_{\substack{a=1\\(a,q)=1}}^q
	\int_{C(a/q)}
	\left|F\left(\alpha, \frac{a}{q}, N\right) \right|^4 d\alpha
	\\&= U_3(N) + U_4(N),
\end{align}
say. By \eqref{eq:F}, we have
\begin{equation} \label{eq:3.3.5}
	U_4(N) \ll \frac{(N \log^2N)^4}{\Omega}
	\sum_{\Delta < q\le \Omega}
	\frac{1}{q^4}
	\ll \frac{N^4 \log^8 N}{\Omega \Delta^3}.
\end{equation}
It remains to estimate $U_3(N)$. By partial summation, we can write $F\left(\alpha, a'/q', N \right)$ as
\begin{align}
	\frac{1}{q'} &\Large(A(a',q')\log^2N + (B(a',q')-2A(a',q'))\log N + 2A(a',q')-B(a',q')
	\\&+C(a',q') \Large)
	\sum_{1\le n\le N} e\left(\left(\alpha-\frac{a'}{q'} \right)n \right)
	- \frac{1}{q'} \int_1^N
	\left(\frac{2A \log \xi}{\xi} + \frac{B-2A}{\xi} \right)
	\sum_{1\le n\le \xi} e\left(\left(\alpha-\frac{a'}{q'} \right)n \right) d\xi.
\end{align}
Thus,
\begin{equation}
	\left|F\left(\alpha, \frac{a'}{q'}, N \right) \right|
	\ll \frac{q'^\epsilon \log^3 N}{q' \left|\sin \pi \left(\alpha - \frac{a'}{q'} \right) \right|}.
\end{equation}
The function $F\left(\alpha, a'/q', N \right)$ has period 1 in $\alpha$, and $|a/q - (a'/q'\pm 1)| \le 1/2$. Thus, $U_3(N)$ is at most
\begin{equation}
	\ll \Delta^2 \log^{12}N 
	\sum_{1\le q\le \Omega}
		\sum_{\substack{a=1\\(a,q)=1}}^q
		\int_{C(a/q)}
		\sum_{1\le q'\le \Delta}
		\sum_{\substack{a'=-q'\\(a',q')=1\\ 0<\left|\frac{a'}{q'} - \frac{a}{q} \right|\le \frac{1}{2}}}^{2q'}
		\frac{q'^\epsilon}{q'^4 \left|\sin \pi \left(\alpha - \frac{a'}{q'} \right) \right|^4} d\alpha.
\end{equation}
By \eqref{eq:3.2.3}, we have, for $\alpha\in C(a/q)$,
\begin{equation}
	\frac{1}{2} \left|\frac{a}{q} - \frac{a'}{q'} \right|
	\le \left|\alpha - \frac{a'}{q'} \right|
	\le \frac{3}{4}
\end{equation}
for $N$ sufficiently large. Hence,
\begin{equation}
	U_3(N)
	\ll \Omega^2 \Delta^2 \log^{12} N
	\sum_{1\le q\le \Omega}
			\sum_{\substack{a=1\\(a,q)=1}}^q
	\sum_{1\le q'\le \Delta} q'^\epsilon
			\sum_{\substack{a'=-q'\\(a',q')=1\\ a'q\neq aq'}}^{2q'}
			\frac{1}{|aq' - qa' |^4}
	\ll \Omega^{2+\epsilon} \Delta^2 \log^{12} N
	\sum_{u=1}^\infty \frac{t(u)}{u^4},
\end{equation}
where $t(u)$ is the number of integer solutions to $|aq'-qa'|=u$ in the range of summation. We have
\begin{equation}
	t(u) \ll \Delta^2 \Omega
\end{equation}
which yields
\begin{equation} \label{eq:3.3.6}
	U_3(N) \ll \Omega^{3+\epsilon} \Delta^4 \log^{12}N.
\end{equation}
From \eqref{eq:3.2.8}, \eqref{eq:3.3.2}, \eqref{eq:3.3.3}, \eqref{eq:3.3.4}, \eqref{eq:3.3.5} and \eqref{eq:3.3.6}, we get the inequality
\begin{equation}
		\sum_{1 \le k\le N-1}
		(V(k,N) - S_\Delta(k,N))^2
		\ll N^\epsilon
		\left(
		N^{11/5} \Omega^{6/5}
		+ \Omega^{12/5} N
		+ \frac{N^{21/5}}{\Omega^2}
		+ \frac{N^3}{\Omega^{8/5}}
		+ \Omega^3 \Delta^4
		+ \frac{N^4}{\Omega \Delta^3}
		\right).
\end{equation}
We now take, for example,
\begin{equation} \label{eq:OmegaDelta}
	\Omega = N^{{25/38}}
	\quad \text{and} \quad
	\Delta = N^{4/19}.
\end{equation}
Then the requirement \eqref{eq:3.2.3} is satisfied, and we have proved
\begin{lem} \label{lemma:3.5.1}
	The inequality
	\begin{equation}
		\sum_{1 \le k\le N-1}
				(V(k,N) - S_\Delta(k,N))^2
				\ll
		N^{299/100}
	\end{equation}
	holds for sufficiently large $N$.
\end{lem}

\section{Proof of theorem \ref{thm:mainresult}}

Let $Q(N)$ denote the sum on the left side of \eqref{eq:2.1.1}. We have
\begin{align} \label{eq:4.2.1}
	Q(N) &= 
	\sum_{1\le \ell\le N}
	\sum_{\substack{1\le n_1,n_2\le N\\ n_1\equiv n_2 (\textrm{mod } \ell)}} \tau_3(n_1) \tau_3(n_2)
	\\&+ \frac{1}{4} N^2 \log^4 N
	\sum_{1\le \ell \le N} \sum_{1\le b\le \ell}
	\tilde{A}(\ell,b)^2
	\\&- N^2 \log^3N \sum_{1\le \ell \le N} \sum_{1\le b\le \ell}
	(\tilde{A}(\ell,b)^2 - \tilde{A}(\ell,b) \tilde{B}(\ell,b))
	\\& + N^2 \log^2 N \sum_{1\le \ell \le N} \sum_{1\le b\le \ell}
	(\tilde{A}(\ell,b)^2 - 2\tilde{A}(\ell,b) \tilde{B} + \tilde{B}(\ell, b)^2)
	\\&+ N^2 \log^2N \sum_{1\le \ell \le N} \sum_{1\le b\le \ell}
	(\tilde{A}(\ell,b)^2 - \tilde{A}(\ell,b) \tilde{B}(\ell,b) + \tilde{A}(\ell,b) \tilde{C}(\ell,b))
	\\&+ 2 N^2 \log^2N \sum_{1\le \ell \le N} \sum_{1\le b\le \ell}
	(\tilde{A}(\ell,b)^2 + \tilde{B}(\ell,b)^2 - 2 \tilde{A}(\ell,b) \tilde{B}(\ell,b) - \tilde{B}(\ell,b) \tilde{C}(\ell,b) + \tilde{A}(\ell,b) \tilde{C}(\ell,b))
	\\&+ N^2 \sum_{1\le \ell \le N} \sum_{1\le b\le \ell}
	(\tilde{A}(\ell,b)^2 + \tilde{B}(\ell,b)^2 + \tilde{C}(\ell,b)^2 - 2 \tilde{A}(\ell,b) \tilde{B}(\ell,b) + 2 \tilde{A}(\ell,b) \tilde{C}(\ell,b) - 2 \tilde{B}(\ell,b) \tilde{C}(\ell,b))
	\\&- N \log^2 N \sum_{1\le \ell \le N} \sum_{1\le b\le \ell}
		\tilde{A}(\ell,b) \sum_{\substack{1\le n\le N\\ n\equiv b (\ell)}} \tau_3(n)
		\\&+ 2 N \log N \sum_{1\le \ell \le N} \sum_{1\le b\le \ell}
		(\tilde{A}(\ell,b) - \tilde{B}(\ell,b))
		\sum_{\substack{1\le n\le N\\ n\equiv b (\ell)}} \tau_3(n)
		\\&+ 2 N \sum_{1\le \ell \le N} \sum_{1\le b\le \ell}
		(\tilde{A}(\ell,b) - \tilde{B}(\ell,b) + \tilde{C}(\ell,b))
		\sum_{\substack{1\le n\le N\\ n\equiv b (\ell)}} \tau_3(n)
	\\ &= Q_1(N) + \cdots + Q_{10}(N),
\end{align}
say. We start with evaluating $Q_1(N)$, whose treatment is the most difficult of the ten. We have
\begin{align} \label{eq:4.5.1}
	Q_1(N) 
	&= N \sum_{1\le n\le N} \tau_3^2(n)
	+ 2 \sum_{1\le \ell \le N-1}\
	\sum_{1\le u\le (N-1)/\ell}\
	\sum_{1\le n\le N-u\ell}\
	\tau_3(n) \tau_3(n+u\ell)
	\\ &= N \sum_{1\le n\le N} \tau_3^2(n)
	+ 2 \sum_{1\le k \le N-1} V(k,N) \tau(k),	
\end{align}
where $V(k,N)$ is given by \eqref{eq:V(k,N)}. Here we have, by Lemma \ref{lemma:tau3squared},
\begin{equation} \label{eq:tau3squared}
	\sum_{n\le N} \tau_3^2(n)
			= \frac{A_3}{8!} N P_8(\log N) + O\left(N^{10/11}\right)
\end{equation}
with $A_3$ and $P_8(\log N)$ given in that lemma. Now, by Lemma \ref{lemma:3.5.1},
\begin{align} \label{eq:4.5.3}
	\sum_{1\le k\le N-1} V(k,N) \tau(k)
	&= \sum_{1\le k\le N-1} S_\Delta(k,N) \tau(k)
	\\&+ O\left\{\left(\sum_{1\le k\le N-1} \tau^2(k) \right)^{1/2}
	\left(\sum_{1\le k\le N-1} (V(k,N) - S_\Delta(k,N))^2 \right)^{1/2} \right\}
	\\ &= \sum_{1\le k\le N-1} S_\Delta(k,N) \tau(k)
	+ O \left(N^{599/300} \right)
	= Q_{11}(N) + O \left(N^{599/300} \right),
\end{align}
say. We now calculate $Q_{11}(N)$. By \eqref{eq:3.2.7}, \eqref{eq:W_q}, and \eqref{eq:T_j}, we have
\begin{equation}
	Q_{11}(N) = 
	\sum_{j=1}^6
	\sum_{1\le q\le \Delta}
	q^{-2} w_j(q)
	\sum_{1\le k\le N-1}
	\tau(k) c_q(k) T_j(k,N).
\end{equation}
If $q=1$, then
\begin{equation}
	c_1(k) = 1,\
	A(1) = 1,\
	B(1) = 3\gamma,\
	C(1) = 3\gamma^2,
\end{equation}
and, hence,
\begin{align} \label{eq:w_j}
	w_1(1) &= \frac{1}{4},\\
	w_2(1) &= \frac{1}{2}(3\gamma-1),\\
	w_3(1) &= (1-3\gamma)^2,\\
	w_4(1) &= -\frac{1}{2}(1-3\gamma+3\gamma^2),\\
	w_5(1) &= (3\gamma-1) (1-3\gamma-3\gamma^2),\\
	w_6(1) &= (1-3\gamma+3\gamma^2)^2.
\end{align}
Thus,
\begin{align} \label{eq:Q11}
	Q_{11}(N)
	&=\sum_{j=1}^6 w_j(1)
	\sum_{1\le k\le N-1} \tau(k) T_j(k,N)
	\\&+ \sum_{j=1}^6
		\sum_{1< q\le \Delta}
		q^{-2} w_j(q)
		\sum_{1\le k\le N-1}
		\tau(k) c_q(k) T_j(k,N).
\end{align}
To calculate the $k$-summations, we need to compute the following sums.
\begin{align} \label{eq:H_j}
	H_1(N) &= \sum_{1\le k\le N-1} \tau(k) \log(N-k),\\
	H_2(N) &= \sum_{1\le k\le N-1} \tau(k) \log^2(N-k),\\
	H_3(X) &= \sum_{1\le k\le X} \tau(k) c_q(k),\\
	H_4(N) &= \sum_{1\le k\le N-1} \tau(k) c_q(k) \log(N-k),\\
	H_5(N) &= \sum_{1\le k\le N-1} \tau(k) c_q(k) \log^2(N-k),\\
	H_6(N) &= \sum_{1\le k\le N-1} k \tau(k) \log(N-k),\\
	H_7(N) &= \sum_{1\le k\le N-1} k \tau(k) \log^2(N-k),\\
	H_8(X) &= \sum_{1\le k\le X} k \tau(k) c_q(k),\\
	H_9(N) &= \sum_{1\le k\le N-1} k \tau(k) c_q(k) \log(N-k),\\
	H_{10}(N) &= \sum_{1\le k\le N-1} k \tau(k) c_q(k) \log^2(N-k).\\
\end{align}
Assume $q>1$. We now compute the first sum in \eqref{eq:H_j}. By partial summation, we have
\begin{align}
	H_1(N) 
	= \int_1^{N-1} \frac{t}{N-t} \log t dt
	+ (2\gamma-1) \int_1^{N-1} \frac{t}{N-t} dt
	+ O(N^{1/2} \log N).
\end{align}
By the first part of Lemma \ref{lem:geometricIntegral}, this is equal to
\begin{equation}
	N\left(\log^2 N - \log N - \frac{\pi^2}{6} + 1 \right)
	+ (2\gamma-1)(N\log N - N) +O(N^{1/2} \log N).
\end{equation}
Thus,
\begin{equation}
	H_1(N) = N\log^2 N
	+ (2\gamma-2) N\log N
	+\left(\frac{\pi^2}{6} - 2\gamma\right) N
	+ O(N^{1/2} \log N).
\end{equation}
Similar, by both parts of Lemma \ref{lem:geometricIntegral}, we get
\begin{align}
	H_2(N) &= \int_1^{N-1} \frac{1}{N-t}
	\left(t\log^2 t + (2\gamma - 2) t\log t + \left(\frac{\pi^2}{6} - 2\gamma \right)t +O(t^{1/2}\log t) \right) dt
	\\&= N\log^3N
	+(2\gamma-4) N\log^2N
	+\left(4-4\gamma-\frac{\pi^2}{6} \right) N \log N
	\\&+\left(2\zeta(3) - 2- (2\gamma-2)\left(\frac{\pi^2}{6}-1 \right) - \frac{\pi^2}{6} + 2\gamma \right) N
	+ O\left(N^{1/2} \log^2 N \right).
\end{align}

We now estimate $H_3(X)$. We have, by definition of the Ramanujan sums,
\begin{equation}
	H_3(X)
	=
	\sum_{\substack{1\le a\le q\\ (a,q)=1}}
	\sum_{ 1 \le k \le X}
	\tau(k) e_q(ak).
\end{equation}
By Lemma \ref{lemma:cqdeltam}, the inner sum is $q^{-1}
X \left(  \log X -2 \log q + 2 \gamma -1 \right)
+
O_{\epsilon} \left( (qX)^{\frac{1}{2} + \epsilon} + q^{1+\epsilon} \right)$. Thus,
\begin{equation}
	H_3(X) 
	=
	\frac{\varphi(q)}{q}
	X
	\left(  \log X -2 \log q + 2 \gamma -1 \right)
	+
	O \left( (q^3X)^{\frac{1}{2} + \epsilon} + q^{2+\epsilon} \right).
\end{equation}
The error term here is negligible. Using the above we get, by partial summation,
\begin{equation}
	H_4(N) = 
	N P_2(\log N)
	+
	O \left( (q^3 N)^{\frac{1}{2} + \epsilon} + q^{2+\epsilon} \right).
\end{equation}
and
\begin{equation}
	H_5(N) = 
	N P_3(\log N)
	+
	O \left( (q^3 N)^{\frac{1}{2} + \epsilon} + q^{2+\epsilon} \right),
\end{equation}
for some polynomials $P_2(\log N)$ and $P_3(\log N)$ of degrees two and three in $\log N$, respectively. Similarly, by partial summation we can easily obtain
\begin{lem}
	\begin{align}
		H_6(N) &=
		\frac{1}{2} (N-1)^2 \log^2 (N-1)
		+\lambda_1 (N-1)^2 \log (N-1)
		+\lambda_2 (N-1)^2
		+ O(N^{3/2} \log N),\\
		H_7(N) &= \frac{1}{2} (N-1)^2 \log^3 (N-1)
		+\lambda_3 (N-1)^2\log^2 (N-1)
		+\lambda_4 (N-1)^2 \log (N-1)
		\\&+ \lambda_5 (N-1)^2
		+ O(N^{3/2} \log^3 N),\\
		H_8(N) 
		&=
		N^2 Q_1(\log N)
		+
		O \left( (q^3 N^2)^{\frac{1}{2} + \epsilon} + N q^{2+\epsilon} \right),
		\\
		H_9(N) &= 
		N^2 Q_2(\log N)
		+
		O \left( (q^3 N^2)^{\frac{1}{2} + \epsilon} + N q^{2+\epsilon} \right), \\
		H_{10}(N) &= 
		N^2 Q_3(\log N)
		+
		O \left( (q^3 N^2)^{\frac{1}{2} + \epsilon} + N q^{2+\epsilon} \right),
	\end{align}
	with numerical constants $\lambda_j$'s and some explicit polynomials $Q_1, Q_2$, and $Q_3$ of degrees one, two, and three, respectively.
	
	Here we have
	\begin{align}
		\lambda_1 &=\gamma - 1/2,\\
		\lambda_2 &= \frac{\pi^2}{12} - \frac{1}{2} \gamma - \frac{3}{4},\\
		\lambda_3 &= \gamma-5/4,
	\end{align}
	etc.
\end{lem}

Collecting the $w_j$'s, $T_j$'s ,and the $H_j$'s above, we deduce the following
\begin{lem} \label{lemma:Q1}
	There is an explicit polynomial $P_5(\log N)$ of degree 5 in $\log N$ such that 
	\begin{equation}
		Q_{11}(N)
		=
		N^2 P_5(\log X)
		+ O \left( N^{\frac{61}{38} + \epsilon} \right).
	\end{equation}
	Consequently, from \eqref{eq:4.5.1}, \eqref{eq:tau3squared}, \eqref{eq:4.5.3}, and \eqref{eq:Q11}, we obtain that
	\begin{equation}
		Q_1(N) = N^2 P_8(\log N) + O(N^{599/300}).
	\end{equation}
\end{lem} 

With more effort, though tedious in details, one can calculate similar asymptotic expansions for $Q_2(N)$ to $Q_{10}(N)$ in \eqref{eq:4.2.1}. However, for our purpose, it suffices to bound the sums $Q_2$-$Q_{10}$ and show that they are smaller than the leading term $N^2 \log^8N$. Indeed, by \eqref{eq:103} and orthogonality of the Ramanujan sum $c_q(b)$, we have that
\begin{equation} \label{eq:748}
	Q_2(N), \cdots, Q_{10}(N) \ll N^2 \log^6 N.
\end{equation}
We demonstrate one such bound for $Q_2(N)$--the other bounds can be obtained similarly. Suppose first that $q=1$. We have, in this case, $\tilde{A}(\ell,b) = \ell^{-1}$ for any $b$, and hence
\begin{equation} \label{eq:838a}
	\sum_{1\le \ell \le N} \sum_{1\le b\le \ell}
		\tilde{A}(\ell,b)^2
	= \sum_{1\le \ell \le N} \sum_{1\le b\le \ell} \ell^{-2}
	\ll \log N.
\end{equation}
Assume next $q_1,q_2>1$. Suppose $(q_1,q_2)=1$. Then
\begin{equation}
	\sum_{1\le b\le \ell}
	c_{q_1}(b) c_{q_2}(b)
	= \sum_{1\le b\le \ell}
		c_{q_1 q_2}(b) \ll q_1q_2.
\end{equation}
From this and \eqref{eq:103}, we get
\begin{align}
	\sum_{1\le b\le \ell} \tilde{A}^2(\ell,b)
	&= \ell^{-2}
	\sum_{q_1\mid \ell} \sum_{q_2\mid \ell}
	q_1^{-1} q_2^{-1}
	\sum_{1\le b\le \ell}
		c_{q_1}(b) c_{q_2}(b)
	\log^2 q_1 \log^2 q_2
	\\&= \ell^{-2}
	\sum_{q_1\mid \ell} \log^2 q_1
	\sum_{q_2\mid \ell} \log^2 q_2
	\ll \ell^{-2} \log^4 \ell
\end{align}
and, hence,
\begin{equation} \label{eq:838b}
	\sum_{1\le \ell \le N} \sum_{1\le b\le \ell}
			\tilde{A}(\ell,b)^2
	\ll \sum_{1\le \ell \le N} \ell^{-2} \log^4 \ell
	\ll \log N.
\end{equation}
It remains to consider the case where $(q_1,q_2) >1$. 
We have
\begin{align}
	\sum_{1 \le b \le \ell}
	c_q(b)^2
	=
	\sum_{d_1 \mid q_1}
	\sum_{d_2 \mid q_2}
	d_1 d_2 \mu(q_1/d_1) \mu(q_2/d_2)
	\sum_{\substack{1 \le b\le \ell\\ d_1 \mid b\\ d_2 \mid b}} 1
	\ll \ell (q_1 q_2)^\epsilon
	+ (q_1 q_2 ^{1+\epsilon}).
\end{align}
Thus,
\begin{equation} \label{eq:838c}
	\sum_{1\le \ell \le N} \sum_{1\le b\le \ell} \tilde{A}^2(\ell,b)
	\ll 
	\sum_{1\le \ell \le N} \ell^{-1} \tau(\ell)
	\ll \log^2 N
\end{equation}
This, together with \eqref{eq:838a} and \eqref{eq:838b}, give that $Q_2(N)$ is at most $O(N^2 \log^6 N)$, verifying \eqref{eq:748} for $Q_2(N)$.

As mentioned before, the estimates in \eqref{eq:748} are crude simply for the purpose of showing they do not contribute to the leading term. It is possible, by procedures analogous to the computations for $Q_1(N)$ and $\sum_k W_q(k,N)$ demonstrated in the proof, to compute explicitly a polynomial $P_6(\log N)$ of degree 6 in $\log N$ such that
\begin{equation}
	Q_2(N) + \cdots + Q_{10}(N) = N^2 P_6(\log N) + O(N^{599/300}).
\end{equation}
We conclude, therefore, that $Q(N)$, which is the left-hand side of \eqref{eq:2.1.1}, is given by
\begin{equation}
	N^2 P_8(\log N) + O(N^{2-1/300}),
\end{equation}
which gives the right-hand side of \eqref{eq:2.1.1}. This completes the proof of the theorem.

\bigskip
\textit{Acknowledgments}. The author thanks Soundararajan for pointing out the reference \cite{HarperSound} at an AIM FRG Seminar, which then motivated him to work on this problem, and B.
Rodgers and J. Stopple for their interests in this project. This work was done while he was
visiting the American Institute of Mathematics, virtually, which he is very grateful for their
hospitality. He would also like to extend his gratitude to the referee for indicating a flaw in a lemma in a previous version and making useful suggestions.

\end{document}